\documentclass{aptpub}
\usepackage{relsize}
\usepackage{algorithm, algpseudocode,amsmath,amssymb,bm,float,mathtools,mathrsfs,multirow,tikz,url}
\usepackage{upgreek} 

\algrenewcommand\algorithmicindent{1.0em}

\usepackage{array}

\newcommand{\CL}{\mathbb{C}}
\newcommand{\RL}{\mathbb{R}}

\newcommand{\Laplace}{\mathcal{L}}
\newcommand{\e}{\mathrm{e}}
\newcommand{\ve}{\bm{\mathrm{e}}} 

\newcommand{\dd}{\,\mathrm{d}}

\newcommand{\oh}{{\mathrm{o}}}
\newcommand{\Oh}{{\mathcal{O}}}
\newcommand{\Exp}{\mathbb{E}}
\newcommand{\dNormal}{\sim\mathrm{N}\,}
\newcommand{\dLNormal}{\sim\mathrm{LN}\,}
\newcommand{\dSLNormal}{\sim\mathrm{SLN}\,}
\newcommand{\Prob}{\mathbb{P}}
\newcommand{\Ind}{\mathbb I}
\newcommand\bfsigma{\bm{\sigma}}
\newcommand\bfSigma{\bm{\Sigma}}
\newcommand\bfLambda{\bm{\Lambda}}

\newcommand\Fset[1]{\mathcal{F}_{\!\mathsmaller{#1}}}

\newcommand\sbullet{
  \tikz[baseline=0] \fill (0,0.05) circle (0.035);
}
\newcommand{\ms}[1]{\mathsmaller{#1}}

\newcommand{\bfA}{\bm{A}}

\newcommand{\bfC}{\bm{C}}
\newcommand{\bfD}{\bm{D}}

\newcommand{\bfF}{\bm{F}}

\newcommand{\bfH}{\bm{H}}
\newcommand{\bfI}{\bm{I}}

\newcommand{\bfM}{\bm{M}}

\newcommand{\bfX}{\bm{X}}

\newcommand{\bfa}{\bm{a}}
\newcommand{\bfb}{\bm{b}}
\newcommand{\bfc}{\bm{c}}

\newcommand{\bfe}{\bm{e}}

\newcommand{\bfq}{\bm{q}}
\newcommand{\bfr}{\bm{r}}

\newcommand{\bfu}{\bm{u}}

\newcommand{\bfw}{\bm{w}}
\newcommand{\bfx}{\bm{x}}
\newcommand{\bfy}{\bm{y}}
\newcommand{\bfz}{\bm{z}}

\newcommand\bfzero{\bm{0}}
\newcommand\bfone{\bm{1}}

\newcommand\bfbeta{\bm{\beta}}

\newcommand\bfmu{\bm{\mu}}

\renewcommand{\det}[1]{\mathrm{det}(#1)}

\newcommand{\sqdet}[1]{\sqrt{\det{#1}}}
\newcommand{\LambertW}{\mathcal{W}}

\newcommand{\diag}{\mathrm{diag}}

\renewcommand{\epsilon}{\varepsilon}

\renewcommand{\pi}{\uppi}

\newcommand{\comp}{\mathsf{c}}
\newcommand{\tr}{\top}
\newcommand{\defeq}{\overset{\scriptscriptstyle\mathrm{def}}{=}}

\makeatletter
\newcommand{\oset}[3][0ex]{
  \mathrel{\mathop{#3}\limits^{
    \vbox to#1{\kern-2\ex@
    \hbox{$\scriptstyle#2$}\vss}}}}
\makeatother

\newcommand{\approxd}{\oset[.2ex]{{\scriptscriptstyle \mathscr D}}{\approx}}
\newcommand{\iidDist}{\overset{\mathrm{i.i.d.}}{\sim}}

\newcommand{\vlog}{\bm{\mathrm{log}}\,}		  

\authornames{Patrick J.\ Laub, S\o ren Asmussen, Jens L.\ Jensen, Leonardo Rojas-Nandayapa}
\shorttitle{Approximating the Laplace transform of the sum of dependent lognormals}

\begin{document}

\title{Approximating the Laplace transform of the sum of dependent lognormals}\\

\authorone[University of Queensland, Aarhus University]{Patrick J.\ Laub}
\authortwo[Aarhus University]{S\O ren Asmussen}
\authorthree[Aarhus University]{Jens L.\ Jensen}
\authorfour[University of Queensland]{Leonardo Rojas-Nandayapa}

\begin{abstract}
Let $(X_1, \dots, X_n)$ be multivariate normal, with mean vector $\bfmu$ and covariance matrix $\bfSigma$,
and $S_n=\e^{X_1}+\cdots+\e^{X_n}$. The Laplace transform ${\cal L}(\theta)=\Exp\e^{-\theta S_n}\propto \int \exp\{-h_\theta(\bfx)\} \dd \bfx$
is represented as $\widetilde {\cal L}(\theta)I(\theta)$, where $\widetilde {\cal L}(\theta)$ is given in closed-form and $I(\theta)$ is the error factor ($\approx 1$). We obtain $\widetilde {\cal L}(\theta)$
by replacing $h_\theta(\bfx)$ with a second order Taylor expansion around its minimiser $\bfx^*$. An algorithm
for calculating the asymptotic expansion of $\bfx^*$ is presented, and it is shown that $I(\theta)\to 1$ as $\theta\to\infty$.
A variety of numerical methods for evaluating $I(\theta)$ are discussed, including Monte Carlo
with importance sampling and quasi-Monte Carlo. Numerical examples (including Laplace transform
inversion for the density of $S_n$) are also given.
\end{abstract}



\addressone{Department of Mathematics, The University of Queensland, Brisbane, Queensland 4072, Australia. Email address: p.laub@[uq.edu.au\textbar math.au.dk]}
\addresstwo{Department of Mathematics, Aarhus University, Ny Munkegade, DK-8000 Aarhus C, Denmark. Email address: \email{asmus@imf.au.dk}}
\addressthree{Department of Mathematics, Aarhus University, Ny Munkegade, DK-8000 Aarhus C, Denmark. Email address: \email{jlj@math.au.dk}}
\addressfour{Department of Mathematics, The University of Queensland, Brisbane, Queensland 4072, Australia. Email address: \email{l.rojasnandayapa@uq.edu.au}}

\section{Introduction}

The lognormal distribution arises in a wide variety of disciplines such as engineering,
economics, insurance, and finance, and is often employed in modeling across the sciences \cite{aitchison1957lognormal,crow1988lognormal,dufresne2009sums,johnson1994continuous,limpert2001log}. It has a natural  multivariate version, namely $(\e^{X_1}, \dots, \e^{X_n}) \dLNormal(\bfmu, \bfSigma)$ when $(X_1, \dots, X_n) \dNormal(\bfmu, \bfSigma)$. In this paper, we consider sums of lognormal random variables, $S_n \defeq \e^{X_1} + \cdots + \e^{X_n}$, where the summands exhibit dependence ($\bfSigma$ is non-diagonal), using the notation that $S_n \dSLNormal(\bfmu, \bfSigma)$. Such sums have many challenging properties. In particular, there are no closed-form expressions for the density $f(x)$ or Laplace transform $\Laplace(\theta)$ of $S_n$. \\

Models using sums of dependent lognormals are widely applicable, though they are particularly important in telecommunications and finance \cite{dufresne2004log,dufresne2009sums}. Indeed, many of the approximations for the Laplace transform of sums of independent lognormals originated from the wireless communications community \cite{beaulieu1995estimating}. This reflects the significance of the SLN distribution within many models, and also that the Laplace transform is of intrinsic interest (engineers frequently work in the Laplace domain). In finance, the value of a portfolio (e.g.\ a collection of stocks) is SLN distributed when using the assumptions of the common Black--Scholes framework. Thus the SLN distribution is central to the pricing of certain options (e.g., Asian and basket) \cite{milevsky1998asian}. Also, financial risk managers require estimates of $f(x)$ across $x \in (0, \Exp[S_n])$ to estimate risk measures such as value-at-risk or expected shortfall.  Estimation of this kind has long been a legal requirement for many large banks, due to the Basel series of regulations (particularly, Basel II and Basel III), so in this context approximating $\Laplace(\theta)$ is useful as a vehicle for computing the density $f(x)$
or the c.d.f. These issues are carefully explained in \cite{duellmann2010regulatory}, \cite{embrechts2014academic}, and the new Chapter 1 in the recently revised volume of McNeil et al.\ \cite{mcneil2015quantitative}. Comprehensive surveys of applications and numerical methods for the LN and SLN distributions
are in \cite{tankov2015tail,asmussen2014laplace,asmussen2015exponential}. \\

There exist many approximations to the density of the SLN distribution. Many approximations work from the premise that a sum of lognormals can be accurately approximated by a single lognormal \cite{beaulieu2004highly}, that is $S_n \approxd L$ where $L\sim \mathrm{LN}(\mu_L, \sigma_L^2)$. We refer to this approach as the \emph{SLN $\approx$ LN approximation}. Some well-known SLN $\approx$ LN approximations are the Wilkinson--Fenton \cite{fenton1960sum} and Schwartz--Yeh \cite{schwartz1982distribution} approaches. These were originally specified for sums of \emph{independent} lognormals, but have since been generalised to the dependent case \cite{abu1994outage}. A more recent procedure (for the independent case) is the minimax approximation of Beaulieu and Xie \cite{beaulieu2004optimal} calculating the values of $\mu_L$ and $\sigma_L$ which minimise the maximum difference between the densities of $S_n$ and $L$.
However, \cite{beaulieu2004optimal} concludes that the approach is inaccurate
in large dimensions or when 
the $X_i$ have significantly different means or standard deviations. Finally, Beaulieu and Rajwani \cite{beaulieu2004highly} describe a family of functions which mimic the characteristics of the SLN distribution function (in the independent case) with some success, i.e., high accuracy and closed-form expressions. \\

Another related avenue of research focuses on the asymptotic behaviour of $f(x)$ in the tails. First, Asmussen and Rojas-Nandayapa \cite{asmussen2008asymptotics} characterised the right tail asymptotics. Next, Gao et al.\ \cite{gao2009asymptotic} gave the asymptotic form of the left tail for $n=2$. Gulisashvili and Tankov \cite{tankov2015tail} then provided the left tail asymptotics for linear combinations of $n \geq 2$ lognormal variables. Yet these asymptotic forms cannot be used to approximate $f(x)$ with precision; to quote \cite[p.\ 29]{tankov2015tail}, ``these formulas are not valid for $x \geq 1$ and in practice have very poor accuracy unless $x$ is much smaller than one''.  Similar numerical experience is reported in
Asmussen et al.\ \cite{asmussen2015exponential}. \\

\enlargethispage{\baselineskip} 
The approach taken here is via the Laplace transform. Accurate estimates for the Laplace transform can be numerically inverted to supply accurate density estimates. Asmussen et al.\ \cite{asmussen2014laplace,asmussen2015exponential} outline a framework to estimate $\Laplace(\theta)$ for $n=1$ using a modified saddlepoint approximation. In their work, the transform is decomposed into $\Laplace(\theta)=\widetilde{\Laplace}(\theta)I(\theta)$, where $\widetilde{\Laplace}(\theta)$ has an explicit form and an efficient Monte Carlo estimator is given for $I(\theta)$. \\

This paper generalises the approach of \cite{asmussen2014laplace,asmussen2015exponential} to arbitrary $n$ and dependence.
The defining integral for the Laplace transform of $S_n$ is
\begin{equation} \label{laplace_def}
	\Laplace(\theta) = \frac{1}{\sqrt{(2\pi)^n \det{\bfSigma}}} \int_{\RL^n} \exp\Big\{ {-}\theta \sum_{i=1}^n \e^{\mu_i} \e^{x_i} -\frac12 \bfx^\tr \bfD \bfx \Big\} \dd \bfx
\end{equation}
where $\bfD \defeq \bfSigma^{-1}$ (assuming $\bfSigma$ to be positive definite so $\bfD$ is well-defined).
Write the integrand as $\exp\{-h_\theta(\bfx)\}$. The idea is then to provide an approximation $\widetilde{\Laplace}(\theta)$ by replacing $h_\theta(\bfx)$
by a second order Taylor expansion around its minimiser $\bfx^*$.  Whereas the minimiser $x^*$ has a simple expression
in terms of the Lambert W function when $n=1$, as in \cite{asmussen2014laplace,asmussen2015exponential}, the situation
is much more complex when $n>1$. As one of our main results we give a limit result for $\bfx^*$ as $\theta\to\infty$.
Further, it is shown that the remainder $I(\theta)$ in the representation $\Laplace(\theta)=\widetilde{\Laplace}(\theta)I(\theta)$ goes to 1, a discussion of efficient Monte Carlo estimators of $I(\theta)$ follows, and numerical results showing the errors of our $\Laplace(\theta)$ and (numerically inverted) $f(x)$ estimators  are given. The paper concludes with an informal discussion regarding estimation of the SLN distribution function $F(x)$, and some closing remarks.

\section{Approximating the Laplace transform}\label{S:ApprL}

Although the definition \eqref{laplace_def} makes sense for all $\theta \in \CL$ with $\Re(\theta) > 0$ (we denote this set as $\CL_\ms{+}$), we will restrict the focus to $\theta \in (0, \infty)$. Of particular interest are the terms in the exponent, which in vector form (see Remark \ref{rem_vector} below) are
\[ h_\theta(\bfx) \defeq \theta (\ve^{\bfmu})^\tr \ve^{\bfx} + \frac12 \bfx^\tr \bfD \bfx \,. \]
An approximation of simple form to $\Laplace(\theta)$ --- written as $\widetilde{\Laplace}(\theta)$ --- is available if $h_\theta(\bfx)$ is replaced by a second order Taylor expansion. The expansion is given in the proposition below. \\

\begin{rem}{\emph{On vector notation.}} \label{rem_vector}
All vectors are considered column vectors. Functions applied elementwise to vectors are written in boldface, such as
$\ve^{\bfx} \defeq (\e^{x_1}, \dots, \, \e^{x_n})^\tr$ and
$\vlog \bfx \defeq (\log x_1, \dots , \log x_n)^\tr$. If a vector is to be elementwise raised to a common power, then the power will be boldface, as in $\bfx^{\bm{k}} \defeq (x_1^k, \dots , x_n^k)^\tr$.
The notation $\bfx \circ \bfy$ denotes elementwise multiplication of vectors. The function $\diag(\cdot)$ converts vectors to matrices and vice versa, like the MATLAB function. $\hfill \diamond$
\end{rem}

\begin{prop}
The second order Taylor expansion of $h_\theta(\bfx)$ about its unique minimiser $\bfx^*$ is
\[ - \Big(\bfone - \frac12 \bfx^* \Big)^\tr \bfD \bfx^* + \frac12 (\bfx - \bfx^*)^\tr (\bfLambda + \bfD) (\bfx - \bfx^*) \]
where $\bfLambda \defeq \theta \, \diag(\ve^{\bfmu + \bfx^*})$.
\end{prop}

\begin{proof}
As $h_\theta(\bfx)$ is strictly convex, a unique minimum exists. Since $\nabla h_\theta(\bfx^*) = \bfzero$, the linear term
vanishes in the Taylor expansion, so we have
\[ h_\theta(\bfx) \approx h_\theta(\bfx^*) + \frac12 (\bfx - \bfx^*)^\tr \bfH (\bfx - \bfx^*) \]
where $\bfH$ is defined as the Hessian $( \partial^2 h_\theta(\bfx) / \, \partial x_i \, \partial x_j )$ evaluated at $\bfx^*$.
To find the value of $\bfH$, we just take derivatives:
\[ \nabla h_\theta(\bfx) = \theta \ve^{\bfmu + \bfx} + \bfD \bfx \,, \qquad \bfH = \bfLambda + \bfD = \bfD(\bfI + \bfSigma \bfLambda) \,. \]
Since $\bfLambda$ and $\bfD$ are both positive definite, so is $\bfH$. Also, $\nabla h_\theta(\bfx^*) = \bfzero$ gives
\begin{equation} \label{grad_zero_subs}
    {-}\theta \ve^{\bfmu + \bfx^*} = \bfD \bfx^* \text{ which implies } {-}\theta (\ve^{\bfmu})^\tr \ve^{\bfx^*} = \bfone^\tr \bfD \bfx^*.
\end{equation}
Therefore the expansion becomes
\begin{align*}
	h_\theta(\bfx) &\approx -\bfone^\tr \bfD \bfx^* + \frac12 (\bfx^*)^\tr \bfD \bfx^* + \frac12 (\bfx - \bfx^*)^\tr (\bfLambda + \bfD) (\bfx - \bfx^*) \\
	&= - \Big(\bfone - \frac12 \bfx^* \Big)^\tr \bfD \bfx^* + \frac12 (\bfx - \bfx^*)^\tr (\bfLambda + \bfD) (\bfx - \bfx^*) \,.
\end{align*}
$\hfill \square$
\end{proof}

This expansion allows $\Laplace(\theta)$ to be approximated as a constant factor $\exp\{-h_\theta(\bfx^*)\}$ times the integral over a normal density (with inverse covariance $\bfLambda + \bfD$), which leads to
\[
	\Laplace(\theta) \approx \widetilde{\Laplace}(\theta) \defeq \frac{1}{\sqrt{\det{\bfI + \bfSigma \bfLambda}}}  \exp\left\{ \Big(\bfone - \frac12 \bfx^* \Big)^\tr \bfD \bfx^* \right\} \,.
\]

We need a suitable error or correction term in order to assess the accuracy of this approximation, so we will decompose the original integral \eqref{laplace_def} into
$\Laplace(\theta) = \widetilde{\Laplace}(\theta) I(\theta)$.
In the integral of \eqref{laplace_def} change variables such that $\bfx = \bfx^* + \bfSigma^{1/2}(\bfI + \bfSigma \bfLambda)^{-1/2} \bfy$. Then by applying \eqref{grad_zero_subs}, multiplying by $\exp\{\bfone^\tr \bfD \bfx^* - \bfone^\tr \bfD \bfx^* \}$, and rearranging, we arrive at
\begin{align*}
	\Laplace(\theta)
	&= \frac{1}{\sqrt{(2\pi)^n\det{\bfI + \bfSigma \bfLambda}}} \int_{\RL^n} \exp\Big\{ {-}\theta (\ve^{\bfmu + \bfx^*})^\tr \ve^{\bfSigma^{1/2}(\bfI + \bfSigma \bfLambda)^{-1/2} \bfy} \\
	&\qquad - \frac12 (\bfSigma^{1/2}(\bfI + \bfSigma \bfLambda)^{-1/2} \bfy)^\tr \bfD (\bfSigma^{1/2}(\bfI + \bfSigma \bfLambda)^{-1/2} \bfy) \Big\}\dd \bfy \\
	&= \widetilde{L}(\theta) I(\theta)
\end{align*}
where
\begin{equation} \label{I_in_propos}
	\begin{split}
	I(\theta) &\defeq \int_{\RL^n} \frac{1}{\sqrt{(2\pi)^n}} \exp\Big\{ (\bfx^*)^\tr \bfD \Big( \ve^{\bfSigma^{1/2} (\bfI+\bfSigma \bfLambda)^{-1/2} \bfy} - \bfone \\
	&\qquad - \bfSigma^{1/2} (\bfI+\bfSigma \bfLambda)^{-1/2} \bfy \Big)  - \frac12 \bfy^\tr (\bfI + \bfSigma \bfLambda)^{-1} \bfy \Big\} \dd \bfy \,.
	\end{split}
\end{equation}
This form may not be particularly elegant. However, it can be rewritten in ways more convenient for Monte Carlo estimation.

\begin{prop} \label{value_of_I}
We have that
\begin{equation} \label{I_as_norm_exps}
	I(\theta) = \Exp\left[ g(\bfSigma^{1/2} (\bfI+\bfSigma \bfLambda)^{-1/2} Z) \right] = \sqrt{\det{\bfI + \bfSigma \bfLambda}} \,\, \Exp\left[ v(\bfSigma^{1/2} Z) \right]
\end{equation}
where
\[ g(\bfu) \defeq \exp\left\{ (\bfx^*)^\tr \bfD (\ve^{\bfu} - \bfone - \bfu) + \frac12 \bfu^\tr \bfSigma \bfLambda (\bfI + \bfSigma \bfLambda)^{-1} \bfu \right\}\,, \]
\[ v(\bfu) \defeq \exp\left\{ (\bfx^*)^\tr \bfD (\ve^{\bfu} - \bfone - \bfu) \right\}\,, \]
and $Z \dNormal(\bfzero, \bfI)$.
\end{prop}

\begin{proof}
To show that $I(\theta)$ can be written as the first expectation in \eqref{I_as_norm_exps}, replace $\frac12 \bfy^\tr (\bfI + \bfSigma \bfLambda)^{-1} \bfy $ in \eqref{I_in_propos} with
\begin{align*}
	&~~- \frac12 \bfy^\tr (\bfI + \bfSigma \bfLambda)^{-1} \bfy - \frac12 \bfy^\tr \bfSigma \bfLambda (\bfI + \bfSigma \bfLambda)^{-1} \bfy + \frac12 \bfy^\tr \bfSigma \bfLambda (\bfI + \bfSigma \bfLambda)^{-1} \bfy \\
	&= -\frac12 \bfy^\tr \bfI \bfy + \frac12 \bfy^\tr \bfSigma \bfLambda (\bfI + \bfSigma \bfLambda)^{-1} \bfy \,.
\end{align*}

To prove $I(\theta)$ equals the second expectation of \eqref{I_as_norm_exps},  change variables in \eqref{I_in_propos} to $\bfz = (\bfI+ \bfSigma \bfLambda)^{-1/2} \bfy$, so
\begin{equation} \label{std_norm_form}
	I(\theta) = \sqdet{\bfI + \bfSigma \bfLambda} \int_{\RL^n} \frac{1}{\sqrt{(2\pi)^n}} \exp\Big\{ (\bfx^*)^\tr \bfD \Big( \ve^{\bfSigma^{1/2}\bfz} - \bfone - \bfSigma^{1/2}\bfz \Big) - \frac12 \bfz^\tr \bfI \bfz \Big\}\dd \bfz \,.
\end{equation}
$\hfill \square$
\end{proof}

\begin{rem} When $n=1$, $\bfSigma = \sigma^2$, and $\mu=0$, \eqref{std_norm_form} becomes
\[
	I(\theta) = \sqrt{ 1 + \theta \sigma^2 \e^{x^*}} \int_{\RL} \frac{1}{\sqrt{2\pi}} \exp\Big\{ \frac{x^*}{\sigma^2} \Big( \e^{\sigma z} - 1 - \sigma z \Big) - \frac12 z^2 \Big\} \dd z \,.
\]
This can be simplified using the \emph{Lambert W} function, denoted $\LambertW(\cdot)$,
which is defined as the solution to the equation $\LambertW(z) \e^{\LambertW(z)} = z$ \cite{corless1996lambertw}. With this we have $x^* = -\LambertW(\theta \sigma^2)$. Also, we can manipulate
\[ \sqrt{ 1 + \theta \sigma^2 \e^{x^*}} =  \sqrt{ 1 - x^* } = \sqrt{ 1 + \LambertW(\theta \sigma^2)} \]
so $I(\theta)$ becomes
\[
	I(\theta) = \sqrt{ 1 + \LambertW(\theta \sigma^2)} \int_{-\infty}^\infty \frac{1}{\sqrt{2\pi}} \exp\Big\{ {-}\frac{\LambertW(\theta \sigma^2)}{\sigma^2} \Big( \e^{\sigma z} - 1 - \sigma z \Big) - \frac12 z^2 \Big\} \dd z
\]
which coincides with the original result of \cite{asmussen2014laplace} equation (2.3). $\hfill \diamond$
\end{rem}

\section{Asymptotic behaviour of the minimiser $\bfx^*$} \label{sec:x_star_asymp_}

We first introduce some notation. For a matrix $\bfX$, we write  $\bfX_{i,\sbullet}$ and $\bfX_{\sbullet,i}$ for the $i$th row and column. Denote the row sums of $\bfD$ as $\bfa = (a_1,\dots,a_n)^\tr$, that is, $a_i=\bfD_{i,\sbullet} \, \bfone$. For sets of indices $\Omega_1$ and $\Omega_2$, then $\bfX_{\Omega_1,\Omega_2}$ denotes the submatrix of $\bfX$ containing row/column pairs in $\{ (u,v) : u \in \Omega_1, v \in \Omega_2\}$.  A shorthand is used for iterated logarithms: $\log_1 \theta \defeq \log \theta$ and $\log_n \theta \defeq \log \log_{n-1} \theta$ for $n \ge 2$ (note that $\log_k \theta$ is undefined for small or negative $\theta$, however this is no problem as we are considering the case $\theta \to \infty$). \\

The approach taken to find $\bfx^*=(x_1^*,\dots,x_n^*)^\tr$ is to set the gradient of $h_\theta(\bfx)$ to $\bfzero$, that is, to solve
\begin{equation} \label{to_solve_xstar}
	\theta \ve^{\bfmu + \bfx^*} + \bfD \bfx^* = \bfzero \,.
\end{equation}
We will show that the asymptotics of the $x_i^*$ are of the form
\begin{equation} \label{general_form}
	x_i^* = \sum_{j=1}^{n} \beta_{i,j} \log_j \theta - \mu_i + c_i + r_i(\theta)
\end{equation}
for some $\bfbeta = (\beta_{i,j}) \in \RL^{n \times n}$, $\bfc = (c_1, \dots, c_n)^\tr \in \RL^n$, and $\bfr(\theta) = (r_1(\theta), \dots, r_n(\theta))^\tr$ where each $r_i(\theta) = \oh(1)$.
Before giving the general result, we consider the special case where all $a_i>0$ since this result and its proof are much simpler.

\begin{prop} \label{all_pos_prop}
If all row sums $\bfD$ are positive then the minimiser $\bfx^*$ takes the form
\begin{equation} \label{all_pos_x_star}
    x_i^* = {-}\log \theta + \log_2 \theta - \mu_i + \log a_i + r_i(\theta)
\end{equation}
where $r_i(\theta) = \Oh(\log_2 \theta / \log \theta) = \oh(1)$ for $1 \leq i  \leq n$, as $\theta \to \infty$.
\end{prop}
\begin{proof}
Inserting \eqref{all_pos_x_star} in \eqref{to_solve_xstar} we find
\[
    \theta \ve^{\bfmu + \bfx^*} + \bfD \bfx^*
    = (\bfa \log \theta) \circ \ve^{\bfr(\theta)} - \bfa \log \theta + \bfa \log_2 \theta - \bfD  \bfmu + \bfD \, \vlog \bfa + \bfD \bfr(\theta) = \bfzero \,.
\]
Looking at these equations we see that we must have
\[
 \limsup_\theta \max_i r_i(\theta) = \liminf_\theta \min_i r_i(\theta) = 0\,,
\]
and to remove the $\log_2 \theta$ term the main term of $r_i(\theta)$ has to be $-\log_2 \theta /\log \theta$. This gives the result of the proposition. $\hfill \square$
\end{proof}

In the general case where some $a_i\le 0$, the asymptotic form of $\bfx^*$ is different from \eqref{all_pos_x_star}
and its derivation is much more intricate.
\begin{thm}\label{Th:24.8a}
There exists a partition of $\{1,\ldots,n\}$ into $\Fset{+}$ and $\Fset{-}$ such that for $i \in \Fset{+}$
\[
x_i^*\ =\ -\log\theta+\log_{k_i}\theta- \mu_i + c_i +\oh(1)
\]
for some $1<k_i\le n$. All $x_i^*$ in $\Fset{-}$ follow the general form of \eqref{general_form}.
In more detail, there exists a partition of $\Fset{-}$ into $\Fset{-}(1)$ and $\Fset{-} \setminus \Fset{-}(1)$, such that if $i\in \Fset{-}(1)$ then $\beta_{i,1} < -1$ and if $i \in \Fset{-} \setminus \Fset{-}(1)$ then
\[ \beta_{i,1} = -1, \beta_{i,2} = \ldots = \beta_{i,k_i-1} = 0, \beta_{i,k_i} < 0 \]
for some $1<k_i\le n$.
Finally we have, writing subscripts $+$ and $-$ for $\Fset{+}$ and $\Fset{-}$, that
$\bfx_\ms{-} = \bfC \bfx_\ms{+} + \oh(1)$ where $\bfC = -\bfD_{\ms{-},\ms{-}}^{-1}\bfD_{\ms{-},\ms{+}}$. The sets $\Fset{+}$, $\Fset{-}$, $\Fset{-}(1)$ and the constants $\beta_{i,j}$, $c_i$, $k_i$ are determined by Algorithm 3.1 below.
\end{thm}
See Remark~\ref{Rem:25.8b} for some further remarks on the role of the signs of the row sums.\\[2mm]

\noindent
{\bf Algorithm 3.1:}
\begin{enumerate}
\item Let $\bfbeta_{\sbullet,1}$ be the value of $\bfw$ that minimises
$\bfw^\tr \bfD \bfw$ over the set $\{ \bfw : w_i \leq -1\}$.
It will be proved in the appendix that the solution has
$\bfD_{i,\sbullet} \, \bfbeta_{\sbullet,1} \leq 0$ when $\beta_{i,1} = -1$ and
$\bfD_{i,\sbullet} \, \bfbeta_{\sbullet,1} = 0$ when $\beta_{i,1} < -1$. Accordingly, we can  partition $\{1,\ldots,n\}$ into the disjoint sets
\[
	\Fset{+}(1) = \emptyset, \quad
	\Fset{*}(1) = \{ i : \bfD_{j,\sbullet} \, \bfbeta_{\sbullet,1} < 0\},
\]
\[
	\Fset{0}(1) = \{ i : \beta_{i,1} = -1, \bfD_{i,\sbullet} \, \bfbeta_{\sbullet,1} = 0\}, \quad
	\Fset{-}(1) = \{ i : \beta_{i,1} < -1\}.
\]
\item For $k = 2, \dots, n$ recursively calculate $\bfbeta_{\sbullet,k}$ as the value
of $\bfw$ that minimises $\bfw^\tr \bfD \bfw$ whilst satisfying
\[
 w_i = 0\ \text{for}\ i\in \Fset{+}(k-1), \quad
 w_i = 1\ \text{for}\ i\in \Fset{*}(k-1),
\]
\[
 w_i \leq 0\ \text{for}\ i\in \Fset{0}(k-1), \quad
 \bfD_{i,\sbullet} \, \bfw = 0\ \text{for}\ i\in \Fset{-}(k-1).
\]
It will be proved in the appendix that the solution has
$\bfD_{i,\sbullet} \, \bfbeta_{\sbullet,k} \leq 0$ for $i \in \Fset{0}(k-1)$,
$\bfD_{i,\sbullet} \, \bfbeta_{\sbullet,k} = 0$ when $\beta_{i,k} < 0$ for $i\in \Fset{0}(k-1)$,
and at least one element of $\Fset{0}(k-1)$ has
$\bfD_{i,\sbullet} \beta_{\sbullet,k} < 0$. This allows us to
create a new partition by
\begin{align*}
	\Fset{+}(k) &= \Fset{+}(k-1) \,\cup\, \Fset{*}(k-1), \\
	\Fset{*}(k) &= \{ i \in \Fset{0}(k-1) : \beta_{i,k}=0, \bfD_{i,\sbullet} \, \bfbeta_{\sbullet,k} < 0 \}, \\
	\Fset{0}(k) &= \{ i \in \Fset{0}(k-1) : \beta_{i,k}=0, \bfD_{i,\sbullet} \, \bfbeta_{\sbullet,k} = 0 \}, \\
	\Fset{-}(k) &= \Fset{-}(k-1) \,\cup\, \{ i \in \Fset{0}(k-1) : \beta_{i,k} < 0 \}.
\end{align*}
Terminate the loop early if $\Fset{0}(k-1) = \emptyset$.
\item Say $\Fset{+} = \Fset{+}(k)$ and $\Fset{-} = \Fset{-}(k)$. For each $i \in \Fset{+}$, let $\ell_i$ to be the index of the first element of $\bfD_{i,\sbullet} \, \bfbeta$ which is negative, and we have $c_i = \log(-\bfD_{i,\sbullet} \, \bfbeta_{\sbullet,\ell_i})$. Determine the remaining elements (using the same subscript shorthand introduced above) by
\begin{equation} \label{consts}
	\bfc_\ms{-} = -\bfD_{\ms{-},\ms{-}}^{-1} \bfD_{\ms{-},\ms{+}} (\bfc_\ms{+} - \bfmu_\ms{+}) + \bfmu_\ms{-}  \,.
\end{equation}
\end{enumerate}

\begin{proof} [Proof of Theorem~{\rm \ref{Th:24.8a}}]
We propose a solution of the form \eqref{general_form} and show that when the $\beta_{i,j}$ are constructed from Algorithm 3.1, the remainder term $r_i$ is $\oh(1)$. \\

The construction allows us to draw the following conclusions for the
$x_i^*$. Let $\Fset{+}$ and $\Fset{-}$ be the sets as defined in Step 3 above.
Consider individually the indices which terminated in the
$\Fset{+}$ and in the $\Fset{-}$ sets. In the first case, there
exists a $k_i$ with $1 < k_i \leq n$ such that
\[
 \beta_{i,j} = \begin{cases}
	-1, & j = 1, \\
	1, & j = k_i, \\
	0, & \text{otherwise},
\end{cases} \quad \text{and} \quad \bfD_{i,\sbullet} \, \bfbeta_{\sbullet,j} =
\begin{cases}
	0, & 1 \le j < k_i - 1, \\
	< 0, & j = k_i - 1.
\end{cases}
\]
Insertion in \eqref{to_solve_xstar} gives
\begin{multline*} 0\ =\
 \theta \e^{\mu_i+x_i^*}+\bfD_{i,\sbullet} \, \bfx^*\ =\
 -\bfD_{i,\sbullet} \, \bfbeta_{\sbullet,k_i-1} \e^{r_i(\theta)} \log_{k_i-1} \theta \\
 +\bfD_{i,\sbullet} \Bigg( \sum_{j=k_i-1}^n\bfbeta_{\sbullet,j}\log_j \theta
 -\bfmu+\bfc+\bfr(\theta) \Bigg)\,,
\end{multline*}
showing that the remainder is $\oh(1)$. \\

In the second case, with $i\in \Fset{+}$,
\[
 \beta_{i,1}<-1 \quad \text{and} \quad
 \bfD_{i,\sbullet} \, \bfbeta_{\sbullet,j}=0,\ 1 \leq j \leq n,
\]
or there exists
$1 < k_i \leq n$ such that
\[ \beta_{i,j} = \begin{cases}
	-1, & j = 1, \\
	0, & 2 \le j < k_i, \\
	< 0, & j = k_i,
\end{cases} \quad \text{and} \quad
	\bfD_{i,\sbullet} \, \bfbeta_{\sbullet,j}=0 \text{ for } 1 \le j \le n\,.
\]
For this case we find
\[
 \theta \e^{\mu_i+x_i^*}+\bfD_{i,\sbullet} \, \bfx^*=
 \oh(1)+\bfD_{i,\sbullet} \, \bfr(\theta),
\]
again showing that the remainder is $\oh(1)$. Lastly, to show $\bfx_\ms{-}$ in terms of $\bfx_\ms{+}$, consider $\theta \ve^{\bfmu_\ms{-} + \bfx_\ms{-}^*} + \bfD_{\ms{-},\ms{+}} \bfx_\ms{+} + \bfD_{\ms{-},\ms{-}} \bfx_\ms{-} = \bfzero$. As $\theta \ve^{\bfmu_\ms{-} + \bfx_\ms{-}^*} = \oh(1)$, then we can see that $\bfx_\ms{-} = -\bfD_{\ms{-},\ms{-}}^{-1} \bfD_{\ms{-},\ms{+}} \bfx_\ms{+} + \oh(1)= \bfC \bfx_\ms{+} + \oh(1)$.
$\hfill \square$
\end{proof}

In some cases above, we have been able to write the constant $c_i$ as an expression involving $\bfD$ and $\bfmu$. For example, in Proposition \ref{all_pos_prop} we have $c_i = \log a_i$, and in Theorem \ref{Th:24.8a} \eqref{consts} gives the value of $c_i$ for $i \in \Fset{-}$. We can show a similar result in the general case for all $i \in \Fset{*}(1)$, that is, for all $i$ where $x_i^* = -\log \theta + \log_2 \theta - \mu_i + c_i + \oh(1)$. \\

Say $\Fset{*} \defeq \Fset{*}(1)$ and $\Fset{\sim} \defeq \Fset{*}^\comp$; in the subscripts below, $*$ and $\sim$ refer to these sets. Since $\bfD$ is regular, so is $\bfD_{\ms{\sim},\ms{\sim}}$. Say that $\overline{\bfD} \defeq \bfD_{\ms{*},\ms{*}}-\bfD_{\ms{*},\ms{\sim}}\bfD_{\ms{\sim},\ms{\sim}}^{-1}\bfD_{\ms{\sim},\ms{*}}$, and denote the corresponding row sums by $\overline{\bfa} = (\overline{a}_i, i \in \Fset{*})$.

\begin{cor}\label{Cor:24.8d}
For all $i \in \Fset{*}$
\[
    x_i^* = {-}\log \theta + \log_2 \theta - \mu_i + \log \overline{a}_i + r_i(\theta)
\]
where $r_i(\theta) = \oh(1)$ and $\overline{a}_i > 0$ as $\theta \to \infty$.
\end{cor}
\begin{proof}
Let $\bfb = -\bfbeta_{\sbullet,1}$. We have
\[
	b_i = \begin{cases}
		1, & i \in \Fset{*}(1) \,\cup\, \Fset{0}(1), \\
		>1, & i \in \Fset{-}(1),
	\end{cases}
	\quad
	\bfD_{i,\sbullet} \bfb = \begin{cases}
		\e^{c_i}, & i \in \Fset{*}(1) = \Fset{*}, \\
		0, & i \in \Fset{0}(1) \,\cup\, \Fset{-}(1) = \Fset{\sim}.
	\end{cases}
\]
Split $\bfD$ according to indices in $\Fset{*}$ and $\Fset{\sim}$, then
\[ \bfD_{\ms{\sim},\ms{*}} \bfb_\ms{*} + \bfD_{\ms{\sim},\ms{\sim}} \bfb_\ms{\sim} = \bfzero \quad \text{and} \quad \bfD_{\ms{*},\ms{*}} \bfb_\ms{*} + \bfD_{\ms{*},\ms{\sim}} \bfb_\ms{\sim}  = \ve^{\bfc_\ms{*}}  > \bfzero\,. \]
The first equation gives $\bfb_{\ms{\sim}} =  - \bfD_{\ms{\sim},\ms{\sim}}^{-1} \bfD_{\ms{\sim},\ms{*}} \bfb_\ms{*}$, and this with the second equation shows
\[ \overline{\bfD} \bfb_\ms{*} = \overline{\bfD} \bfone = \overline{\bfa} = \ve^{\bfc_\ms{*}} > \bfzero \,,\]
thus $\overline{\bfD}$ has all row sums positive and
$\bfc_\ms{*} = \vlog(\overline{\bfD}\bfb_\ms{*}) = \vlog(\overline{\bfa})$. $\hfill \square$
\end{proof}

There are some simple forms of $\bfSigma$ which fall into the case where all $a_i > 0$. These include the case where all diagonal elements of $\bfSigma$ are identical, and all non-diagonal elements are identical. Note, by positive definiteness of $\bfSigma$ we must have at least one row sum positive. Also, if $X_1, \dots, X_n$ is an AR(1) process, then the resulting covariance matrix would have all $a_i > 0$. Meanwhile, cases where $\exists \, a_i \leq 0$ are not difficult to find. For the case $n=2$ with variances $\sigma_1^2 \leq \sigma_2^2$ and correlation $\rho$, a simple calculation gives that both row sums are positive when $\rho < \sigma_1 / \sigma_2$, and one is negative when $\rho > \sigma_1 / \sigma_2$ (see Gao et al.\ \cite{gao2009asymptotic} for the expansion of $f(x)$ as $x \downarrow 0$ for these cases). We now list a couple of examples of asymptotic forms of $\bfx^*$ for specific $\bfmu$ and $\bfSigma$ which have non-positive row sums.

\begin{ex}
Consider $\bfmu=(-10,0, 10)^\tr$ and
\[ \bfSigma \ =\ \left(\begin{array}{ccc}
            0.5 & 1 & 2 \\
            1 & 3 & 4 \\
            2 & 4 & 10
        \end{array}\right)\,, \quad
        \bfD \ =\ \left(\begin{array}{ccc}
            14 & -2 & -2 \\
            -2 & 1 & 0 \\
            -2 & 0 & 0.5
        \end{array}\right) . \]
Implementing the algorithm gives that
\begin{align*}
    x_1^* &= -\log \theta + \log_2 \theta +(10 +\log 2) + \oh(1) \,, \\
    x_2^* &= -2 \log \theta + 2 \log_2 \theta + (20 + 2\log 2) + \oh(1) \,, \\
    x_3^* &= -4 \log \theta + 4 \log_2 \theta + (40 + 4\log 2) + \oh(1) \,,
\end{align*}
and
\[
(\bfbeta \,|\, \bfc-\bfmu) = \left(\begin{array}{ccc|c}
            -1 & 1 & 0 & 10.69 \\
            -2 & 2 & 0 & 21.39 \\
            -4 & 4 & 0 & 42.77
        \end{array}\right) ,\,
    \bfD (\bfbeta \,|\, \bfc-\bfmu) = \left(\begin{array}{ccc|c}
            -2 & * & * & * \\
            0 & 0 & 0 & 0 \\
            0 & 0 & 0 & 0
    \end{array}\right)
\]
(where unimportant values of $\bfD (\bfbeta \,|\, \bfc-\bfmu)$ are replaced by stars).
$\hfill \square$
\end{ex}

\begin{ex}
Consider $\bfmu=(1,2,3)^\tr$ and
\[\bfSigma = \left(\begin{array}{ccc}
            0.4545 & 0.4545 & 0.4545 \\
            0.4545 & 1.7204 & 1.8470 \\
            0.4545 & 1.8470 & 2.9862
        \end{array}\right)\,, \quad \bfD \ =\ \left(\begin{array}{ccc}
            3 & -0.9 & 0.1 \\
            -0.9 & 2 & -1.1 \\
            0.1 & -1.1 & 1
        \end{array}\right) . \]
Implementing the algorithm gives that
\begin{align*}
    x_1^* &= -\log \theta + \log_2 \theta -1 +\log 2.2 + \oh(1) \,, \\
    x_2^* &= -\log \theta + \log_3 \theta -2 +\log 0.79 + \oh(1) \,, \\
    x_3^* &= -\log \theta -0.1\log_2 \theta + 1.1 \log_3 \theta -3 + c_3 + \oh(1) \,,
\end{align*}
where $c_{3} = 0.9 -0.1 \log 2.2 + 1.1 \log 0.79$, and
\[
(\bfbeta \,|\, \bfc-\bfmu) = \left(\begin{array}{ccc|c}
            -1 & 1 & 0 & -0.2 \\
            -1 & 0 & 1 & -2.2 \\
            -1 & -0.1 & 1.1 & -2.4
        \end{array}\right),\,
    \bfD (\bfbeta \,|\, \bfc-\bfmu) = \left(\begin{array}{ccc|c}
            -2.2 & * & *& * \\
            0 & -0.79 & * & * \\
            0 & 0 & 0 & 0
    \end{array}\right)
.
\]
$\hfill \square$
\end{ex}

\begin{rem}\label{Rem:25.8b} \rm The importance of the sign of the row sums of $\bfD$, as illustrated by Proposition \ref{all_pos_prop}, perplexed us for quite some time. However Gulisashvili and Tankov \cite{tankov2015tail} describe an interesting link between the row sums and the \emph{minimum variance portfolio}. They show that the leading asymptotic term of $\Prob(S_n < x)$ as $x \downarrow 0$ depends upon
\[ \overline{\bfw}^\tr \bfSigma \, \overline{\bfw} = \min_{\bfw \in \Delta} \bfw^\tr \bfSigma \, \bfw \,, \text{ where }
\Delta \defeq \{ \bfw : \sum_i w_i = 1, w_i \geq 0 \} \,. \]
The $i$ in which $\overline{w}_i > 0$ indicate which summands in $S_n$ have the `least variance'. These summands are asymptotically important in the left tail, as they will struggle the most to take very small values. Seen from the viewpoint of modern portfolio theory \cite{markowitz1952portfolio} then the solution $\overline{\bfw}$ is viewed as the optimal portfolio weights to create the minimum variance portfolio. When all $a_i > 0$ then $\overline{w}_i = a_i / \sum_{j=1}^n a_j$ which represents full diversification. However when assets become highly correlated (meaning that some $\bfD$ row sums are non-positive) then $\exists \, \overline{w}_i = 0$, i.e., some assets are ignored. Thus the asymptotics are qualitatively different when the signs of the row sums change. The exact point where an asset's optimal weight becomes 0 occurs when $a_i = 0$, and this phase change produces a unique and convoluted asymptotic form. As $\Laplace(\theta)$ as $\theta \to \infty$ is related to $\Prob(S_n < x)$ as $x \downarrow 0$ then the behaviour of $\bfx^*$ is explained. $\hfill \diamond$
\end{rem}

For applications we will need to find $\bfx^*$ for a large number of $\theta$ numerically. The results above give a sensible starting point for an iterative solver, such as Newton--Raphson. Another option is based on the following formulation. Let $\bfA \defeq \bfD - \diag(\bfD)$ and write the defining equation as
\[ \theta \ve^{\bfmu + \bfx^*} + \diag(\bfD) \, \bfx^* = - \bfA \bfx^* \,. \]
For each row $i$, all $x_i^*$ are now on the left-hand side. Using properties of the Lambert W function we see that
\[ x_i^* = -\LambertW\left( \frac{\theta \e^{\mu_i}}{D_{i,i}} \exp\left\{ - \frac{\bfA_{i,\sbullet} \, \bfx^*}{D_{i,i}} \right\} \right)  - \frac{\bfA_{i,\sbullet} \, \bfx^*}{D_{i,i}} \,. \]
One can use this to perform a componentwise fixed point iteration as an alternative to the Newton--Raphson scheme.

\section{Asymptotic behaviour of $I(\theta)$}

In order to discuss $I(\theta)$ as $\theta \to \infty$ we will consider it in a  form different from
Section~\ref{S:ApprL}. Define
$ \bfsigma \defeq \diag(\bfLambda + \bfD)^{-1/2} \in (0,\infty)^n$ and $\bfM \defeq \diag(\bfsigma) \, (\bfLambda + \bfD) \, \diag(\bfsigma) \in \RL^{n \times n}$.
In \eqref{I_in_propos}, substitute $\bfSigma^{1/2} (\bfI+\bfSigma \bfLambda)^{-1/2} \bfy = (\bfsigma \circ \bfz)$, so
\begin{align} \label{new_I}
	I(\theta) &= \int_{\RL^n} \frac{\exp\{{-}\frac12 \bfz^\tr \bfM \bfz\}}{\sqrt{(2\pi)^n\det{\bfM^{-1}}}}
	\exp\Big\{
    {-}
    \theta (\ve^{\bfmu + \bfx^*})^\tr  \Big( \ve^{\bfsigma \circ \bfz} - \bfone - \bfsigma \circ \bfz - \frac12(\bfsigma \circ \bfz)^{\bm{2}} \Big)
    \Big\}\dd \bfz \,.
\end{align}

The limit of this integrand is the density of a multivariate normal distribution, which when integrated is 1. To see this, consider the following. As $\theta \to \infty$ we have $\sigma_i \to \infty$ or $\sigma_i \to D_{i,i} > 0$, so taking $\ell \in (2,\infty)$ means
\begin{equation} \label{oh_one}
	\theta \e^{\mu_i + x_i^*} \sigma_i^\ell = \theta \e^{\mu_i + x_i^*} (\theta \e^{\mu_i + x_i^*} + D_{i,i})^{-\ell/2} = \oh(1) \,.
\end{equation}
Consider the second exponent of \eqref{new_I}. For fixed $\bfz$, $\e^{\sigma_i z_i} - 1 - \sigma_i z_i - \frac12 \sigma_i^2 z_i^2 = \Oh(\sigma_i^3)$, and since $\theta \e^{\mu_i + x_i^*} \sigma_i^3 = \oh(1)$ by \eqref{oh_one} we have
\begin{equation} \label{second_exp_to_zero}
	\theta (\ve^{\bfmu + \bfx^*})^\tr  \Big( \ve^{\bfsigma \circ \bfz} - \bfone - \bfsigma \circ \bfz - \frac12(\bfsigma \circ \bfz)^{\bm{2}} \Big) = \oh(1) \,.
\end{equation}

Finally, we consider $\bfM$ as $\theta \to \infty$. Say that $n_\ms{+} \defeq |\Fset{+}|$ and assume that these are the first $n_\ms{+}$ indices. We can then write that $\bfM \to \bfM^* \defeq \diag(\bfI_{n_\ms{+}}, \bfF)$ where this $\bfF$ is the bottom-right submatrix of size $(n-n_\ms{+}) \times (n-n_\ms{+})$ of the inverted correlation matrix implied by $\bfSigma$. The $\bfM$ matrices are positive definite for all $\theta\in(0,\infty]$, thus the limiting form of the integrand in \eqref{new_I} is a nondegenerate multivariate normal density.

\begin{prop} $\displaystyle \lim_{\theta \to \infty} I(\theta) = 1$.
\end{prop}
\begin{proof}
We use the dominated convergence theorem. By \eqref{second_exp_to_zero} and the paragraph which follows that equation, the exponent of the integrand is bounded by a constant $g_1$ for $||\bfz|| < 1$, say, and that the exponent is below $-g_2 ||\bfz||$ otherwise ($g_2 > 0$), for $\theta > \theta_0$, say. The latter comes from the positive definiteness of $\bfM^*$, the convergence of $\bfM$ to $\bfM^*$, and the convergence of \eqref{second_exp_to_zero}. Next, convexity implies that the exponent is bounded by $−g_2||\bfz||$ for $||\bfz|| > 1$. In total we have the bound
\[ \exp \Big\{ g_1 \Ind_{\{||\bfz|| \leq 1\}} - g_2 ||\bfz|| \Ind_{\{||\bfz|| > 1\}} \Big\} \,, \]
which is an integrable function. Thus the conditions for dominated convergence are satisfied and we can safely
switch the limit and integral to obtain $I(\theta) \to 1$.
$\hfill \square$
\end{proof}

\section{Estimators of $\Laplace(\theta)$ and $I(\theta)$}

The simplest approach is to numerically integrate the original expression in \eqref{laplace_def}. This approach is used as a baseline against which the following estimators are compared (the approach can, however,  be slow or impossible for large $n$). The next na{\"i}ve approach is to estimate the expectation $\Exp[\e^{-\theta S_n}]$ by crude Monte Carlo (CMC). This would involve simulating random vectors $X_1, \dots, X_R \iidDist \mathrm{LN}(\bfmu, \bfSigma)$, with $X_r = (X_{r,1}, \dots, X_{r,n})$, and computing
\[ \widehat{\Laplace}_{\mathrm{CMC}}(\theta) \defeq \frac{1}{R} \sum_{r=1}^R \exp\Big\{-\theta \sum_{i=1}^n X_{r,i}\Big\}\,. \]
However this estimator is not efficient for large $\theta$, and rare-event simulation techniques are required. \\

Given the decomposition of $\Laplace(\theta) = \widetilde{\Laplace}(\theta) I(\theta)$, then some more accurate estimators can be assessed. Simply using $\widetilde{\Laplace}(\theta)$ gives a biased estimator (which is fast and deterministic) for the transform, however the bias is decreased by estimating $I(\theta)$ with Monte Carlo integration. Proposition \ref{value_of_I} gives two probabilistic representations of $I(\theta)$. We expect the CMC estimator of the first --- $\Exp[g(\bfSigma^{1/2} (\bfI+\bfSigma \bfLambda)^{-1/2}Z) ]$ --- to exhibit infinite variance as $\theta \to \infty$ as this has been proven for $n=1$ in \cite{asmussen2014laplace}. Therefore this estimator does not seem promising. The second estimator --- $\sqrt{\det{\bfI + \bfSigma \bfLambda}} \,\, \Exp[ v(\bfSigma^{1/2} Z) ]$ --- can be viewed as the first estimator after importance sampling has been applied, so we focus upon this. Taking $Z_1, \dots, Z_R \iidDist \mathrm{N}(\bfzero, \bfSigma)$, then
\[ \widehat{\Laplace}_{\mathrm{IS}}(\theta) \defeq \frac{1}{R}  \exp\left\{ \Big(\bfone - \frac12 \bfx^* \Big)^\tr \bfD \bfx^* \right\}  \sum_{r=1}^R \exp\left\{ (\bfx^*)^\tr \bfD (\ve^{Z_r} - \bfone - Z_r) \right\} \,. \]

Many variance reduction techniques can be applied to increase the efficiency of these estimators. The effect of including control variates into $\widehat{\Laplace}_{\mathrm{IS}}(\theta)$ was considered, using the control variate $(\bfx^*)^\tr \bfD Z_r^{\bm{2}}$ (note the elementwise square). The variance reduction achieved was small considering the large overhead of computing the variates (and their expectations) so these results have been omitted. Lastly, we considered an estimator based on the Gumbel distribution. Say that $G = (G_1, \dots, G_n)$ is a vector of i.i.d.\ standard Gumbel random variables, that is, $\Prob(G_r < x) = \exp\{ -\e^{-x} \}$ for $x \in \RL$. Then $\Laplace(\theta)$ can be rewritten as an integral over the density of a vector of standard Gumbel random variables. This estimator was quite accurate, though it had higher relative error and variance than the estimators based on $\widehat{\Laplace}_{\mathrm{IS}}(\theta)$ so it too has been excluded from the results. \\

The final two variance reduction techniques investigated were \emph{common random numbers} and \emph{quasi-Monte Carlo} applied to $\widehat{\Laplace}_{\mathrm{IS}}(\theta)$; for a detailed explanation of these techniques see \cite{glasserman2003monte} or \cite{asmussen2007stochastic}. Both individually achieved significant variance reduction, and together provided the best estimator. Specifically,
\[ \widehat{\Laplace}_{\mathrm{Q}}(\theta) \defeq \frac{1}{R} \exp\Big\{ \Big(\bfone - \frac12 \bfx^* \Big)^\tr \bfD \bfx^* \Big\}  \sum_{r=1}^R \exp\left\{ (\bfx^*)^\tr \bfD (\ve^{\bfq_r} - \bfone - \bfq_r) \right\} \]
where $\bfq_r \defeq \bfSigma^{1/2} \mathrm{\bf\Phi}^{-1}(\bfu_r)$, using $\mathrm{\bf\Phi}^{-1}(\cdot)$ as the (elementwise) standard normal inverse c.d.f., and where $\{\bfu_1, \bfu_2, \dots\}$ is the $n$ dimensional Sobol sequence started at the same point for every $\theta$. Therefore, $\widehat{\Laplace}_{\mathrm{Q}}(\theta)$ is deterministic (for a fixed $R$ and $\theta$), and using this scheme is therefore a kind of numerical quadrature. More sophisticated adaptive quadrature methods
could possibly be applied.

\section{Numerical Results}

Relative errors are given for the main estimators of $\Laplace(\theta)$ in the table below. In all estimators the smoothing technique of using common random variables is employed, and all estimators are compared against numerical integration of the relevant integrals to 15 significant digits.
\setlength\extrarowheight{3pt}
\begin{table}[H]
\centering
\begin{tabular}{cccccc|}
\cline{2-6}

\multicolumn{1}{c|}{$\theta$}& 100 & 2{,}500 & 5{,}000 & 7{,}500 & 10{,}000 \\ \hline

\multicolumn{1}{|c|}{$\widetilde{\Laplace}$}&
\multicolumn{1}{c|}{-9.89e-3 }&
\multicolumn{1}{c|}{-1.27e-2 }&
\multicolumn{1}{c|}{-1.28e-2 }&
\multicolumn{1}{c|}{-1.27e-2 }&
\multicolumn{1}{c|}{-1.27e-2 }
\\ \hline

\multicolumn{1}{|c|}{$\widehat{\Laplace}_{\mathrm{CMC}}$}& \multicolumn{1}{c|}{1.29e-2 }& \multicolumn{1}{c|}{*}& \multicolumn{1}{c|}{*}& \multicolumn{1}{c|}{*}& \multicolumn{1}{c|}{*} \\ \hline

\multicolumn{1}{|c|}{$\widehat{\Laplace}_{\mathrm{IS}}$}& \multicolumn{1}{c|}{3.36e-4 }& \multicolumn{1}{c|}{2.96e-4 }& \multicolumn{1}{c|}{2.57e-4 }& \multicolumn{1}{c|}{2.31e-4 }& \multicolumn{1}{c|}{2.11e-4 } \\ \hline

\multicolumn{1}{|c|}{$\widehat{\Laplace}_{\mathrm{Q}}$}& \multicolumn{1}{c|}{-3.19e-6 }& \multicolumn{1}{c|}{-5.03e-6 }& \multicolumn{1}{c|}{-5.31e-6}& \multicolumn{1}{c|}{-5.56e-6 }& \multicolumn{1}{c|}{-5.98e-6 } \\ \hline
\end{tabular}
\caption{Relative error for various approximations of $\Laplace(\theta)$ for $\bfmu = \bfzero$, $\bfSigma = [1, 0.5; 0.5, 1]$. The number of Monte Carlo replications $R$ used is $10^6$. Note: a * indicates that the CMC estimator simply gave an estimate of 0.}
\end{table}

Also, the p.d.f.\ of $S_n$ can be estimated by numerical inversion of the Laplace transform. As the approximations of $\Laplace(\theta)$ above are only valid for $\theta \in (0, \infty)$, not $\theta \in \CL_\ms{+}$, then this restricts the options for Laplace transform inversion algorithms. The Gaver--Stehfest algorithm \cite{stehfest1970algorithm} and so-called power algorithms \cite{avdis2007power} can be used. We report on the results of using the Gaver--Stehfest algorithm as implemented by Mallet \cite{mallet1985numerical}. \\

Other options for estimating $f(x)$ include numerically integrating the convolution equation (typically this is only viable for small $n$), the conditional Monte Carlo method (as in Example 4.3 on page 146 of \cite{asmussen2007stochastic}), and kernel density estimation. The following estimators are reported: the conditional Monte Carlo estimator $\widehat{f}_{\mathrm{Cond}}$, $\widetilde{f} \defeq \Laplace^{-1}(\widetilde{\Laplace}(\cdot))$, $\widehat{f}_{\mathrm{IS}} \defeq \Laplace^{-1}(\widehat{\Laplace}_{\mathrm{IS}}(\cdot))$, and $\widehat{f}_{\mathrm{Q}} \defeq \Laplace^{-1}(\widehat{\Laplace}_{\mathrm{Q}}(\cdot))$.

\begin{table}[H]
\centering

\begin{tabular}{cccccc|}
\cline{2-6}
\multicolumn{1}{c|}{$x$}& 0.01 & 1 & 1.5 & 2 & 3 \\ \hline
\multicolumn{1}{|c|}{$\widehat{f}_{\mathrm{Cond}}$}& \multicolumn{1}{c|}{-1.17e-1 }& \multicolumn{1}{c|}{2.20e-2 }& \multicolumn{1}{c|}{3.72e-3 }& \multicolumn{1}{c|}{5.21e-3 }& \multicolumn{1}{c|}{-4.60e-3 } \\ \hline

\multicolumn{1}{|c|}{$\widetilde{f}$}& \multicolumn{1}{c|}{-7.03e-3}& \multicolumn{1}{c|}{2.56e-2 }& \multicolumn{1}{c|}{1.79e-2 }& \multicolumn{1}{c|}{6.00e-2 }& \multicolumn{1}{c|}{3.82e-2 } \\ \hline

\multicolumn{1}{|c|}{$\widehat{f}_{\mathrm{IS}}$}& \multicolumn{1}{c|}{1.94e-3 }& \multicolumn{1}{c|}{1.43e-2 }& \multicolumn{1}{c|}{-6.13e-3}& \multicolumn{1}{c|}{4.00e-2 }& \multicolumn{1}{c|}{3.68e-3 } \\ \hline

\multicolumn{1}{|c|}{$\widehat{f}_{\mathrm{Q}}$}& \multicolumn{1}{c|}{2.90e-4 }& \multicolumn{1}{c|}{1.11e-2 }& \multicolumn{1}{c|}{-9.04e-3}& \multicolumn{1}{c|}{3.70e-2 }& \multicolumn{1}{c|}{2.44e-3 } \\ \hline

\end{tabular}
\caption{Relative errors for estimators of $f(x)$ for $\bfmu = \bfzero$ and $\bfSigma = [1, 0.5; 0.5, 1]$. The number of Monte Carlo repetitions for each $x$ is $R=10^4$ for $\widehat{f}_{\mathrm{Cond}}$, $\widehat{f}_{\mathrm{IS}} $ and $\widehat{f}_{\mathrm{Q}} $.}
\end{table}

The numerically inverted Laplace transforms are surprisingly accurate. Using common random numbers for the $\Laplace(\theta)$ estimators was necessary, otherwise the inversion algorithms became confused by the non-smooth input. The precision of the inversion algorithms cannot be arbitrarily increased when using standard double floating point arithmetic \cite{abate2006unified}, so the software suite Mathematica was used. Yet this did not solve the problem of the Gaver--Stehfest algorithm becoming unstable (and very slow) when trying to increase the desired precision. Also, the inversion results became markedly poorer when $f(x)$ exhibited high kurtosis (i.e., when $\det{\bfSigma}$ became small).

\section{Closing Remarks}\label{S:Remarks}

The estimators above give an accurate, relatively simple, and computationally swift method of computing the Laplace transform of the sum of dependent lognormals. We have shown that the approximation's error diminishes to zero ($I(\theta) \to 1$) as $\theta \to \infty$, and that it is still accurate for small values of $\theta$. One can find $\bfx^*$ --- for each $\theta$ examined --- using a Newton--Raphson scheme and Section \ref{sec:x_star_asymp_} gives an accurate starting value for the iterations.

\appendix

\section{Remaining steps in the proof of Theorem~\ref{Th:24.8a}}
\begin{proof}
First we note that all the minimisations are convex problems and therefore have unique solutions. \\

For the initial step of the algorithm let $\overline \bfw$ be the solution of the minimisation problem and let $\bfe_i$ be the vector with 1 at coordinate $i$ and zero at the other coordinates. Then
$g_i(\epsilon)=(\overline \bfw+\epsilon \bfe_i)^\tr\bfD(\overline \bfw+\epsilon \bfe_i)$ is minimised at $\epsilon=0$. When $\overline w_i<-1$ the vector $\overline \bfw+\epsilon \bfe_i$ is in the search set for all $\epsilon$ small. We therefore have $g_i^\prime(0)=0$ which gives $\bfD_{i,\sbullet}\overline \bfw=0$.
When $\overline w_i=-1$ the vector $\overline \bfw+\epsilon \bfe_i$ is
in the search set only for nonpositive values of $\epsilon$.
This implies $g_i^\prime(0)\leq 0$ giving $\bfD_{i,\sbullet}\overline \bfw\leq 0$. \\

For the general recursive step we let $\bfu=\bfw_{\Fset{0}(k-1)}$ and express
$\bfw_{\Fset{-}(k-1)}$ in terms of $\bfu$ from the equations
$\bfD_{i,\sbullet}\bfw=0$, $i\in \Fset{-}(k-1)$. The derivative of $\bfw^\tr \bfD \bfw$ with respect to
$u_i$ ($i$ being the index inherited from $\bfw$) is then
$2\bfD_{i,\sbullet}\bfw+2(\partial \bfw_{\Fset{-}(k-1)}/\partial u_i)\bfD_{\Fset{-}(k-1)}\bfw=
2\bfD_{i,\sbullet}\bfw$. As above we find that the derivatives of $\bfw^\tr \bfD \bfw$
with respect to $u_i$ at the minimising point is
zero when $u_i<0$ and less than or equal to zero when $u_i=0$. \\

What is left to prove is that $\Fset{0}(k)$ always has at least one element
with $\bfD_{i,\sbullet}\bfbeta_{\sbullet,k+1}<0$. To this end define
$d_1=-\bfbeta_{\sbullet,1}$ and $d_k=d_{k-1}-\bfbeta_{\sbullet,k}$ for $k>1$.
From the properties of $\bfbeta$ we find
\begin{align*}
 & d_{\Fset{+}(k),k}=0;\quad
 d_{\Fset{*}(k),k}=1\ \text{and}\ \bfD_{\Fset{*}(k)}d_k>0;
 \\ &
 d_{\Fset{0}(k),k}=1\ \text{and}\ \bfD_{\Fset{0}(k)}d_k=0;
 \quad
 \bfD_{\Fset{-}(k)}d_k=0.
\end{align*}
Assume now that $\bfD_{i,\sbullet}\bfbeta_{\sbullet,k+1}=0$ for all
$i\in \Fset{0}(k)$. We show that this leads to a contradiction.
Using the assumption $\bfbeta_{\sbullet,k+1}$ has the properties
\begin{align*}
 & \bfbeta_{\Fset{+}(k),k+1}=0;\quad
 \bfbeta_{\Fset{*}(k),k+1}=1;
 \\ &
 \bfbeta_{\Fset{0}(k),k+1}\leq 0\ \text{and}\ \bfD_{\Fset{0}(k)}\bfbeta_{\sbullet,k+1}=0;
 \quad
 \bfD_{\Fset{-}(k)}\bfbeta_{\sbullet,k+1}=0.
\end{align*}
Combining the two displays we have
\[
 \bfD_{\Fset{0}(k)}d_k=\bfD_{\Fset{0}(k)}\bfbeta_{\sbullet,k+1}\ \text{and}\
 \bfD_{\Fset{-}(k)}d_k=\bfD_{\Fset{-}(k)}\bfbeta_{\sbullet,k+1}.
\]
Since $d_k$ and $\bfbeta_{\sbullet,k+1}$ are identical on $\Fset{+}(k-1)$ and
$\Fset{*}(k-1)$ the equations reduce to
\[
 \bfD_0
  \left(\begin{array}{c}
 d_{\Fset{0}(k),k} \\ d_{\Fset{-}(k),k}
 \end{array}
 \right)
 =
 \bfD_0
  \left(\begin{array}{c}
 \bfbeta_{\Fset{0}(k),k} \\ \bfbeta_{\Fset{-}(k),k}
 \end{array}
 \right),
 \ \text{where}\
 \bfD_0=\left(\begin{array}{cc}
 \bfD_{\Fset{0}(k),\Fset{0}(k)} & \bfD_{\Fset{0}(k),\Fset{-}(k)} \\
 \bfD_{\Fset{-}(k),\Fset{0}(k)} & \bfD_{\Fset{-}(k),\Fset{-}(k)}
 \end{array}
 \right).
\]
Since the matrix $\bfD_0$ is positive definite and since
$d_{\Fset{0}(k),k}\neq \bfbeta_{\Fset{0}(k),k}$ we have reached a contradiction. \hfill $\square$
\end{proof}

\bibliographystyle{apt}
\bibliography{dependence}

\end{document}